\documentclass[11pt]{amsart}
\usepackage{dsfont}
\usepackage{amssymb,amsthm,amsmath}
\usepackage{xcolor}
\usepackage[colorlinks, pdftex]{hyperref}

\theoremstyle{plain}
\newtheorem{prop}{Proposition}[section]
\newtheorem{lem}[prop]{Lemma}

\newtheorem{thm}[prop]{Theorem}

\newtheorem{cor}[prop]{Corollary}

\theoremstyle{definition}

\theoremstyle{remark}
\newtheorem{remark}[prop]{Remark}

\numberwithin{equation}{section}

\DeclareMathOperator{\Bow}{Bow}
\DeclareMathOperator{\SL}{SL}

\DeclareMathOperator{\SO}{SO}

\DeclareMathOperator{\diag}{diag}

\newcommand\hs{homogeneous space}

\newif\ifdraft\drafttrue
\draftfalse

\newcommand\eq[2]{{\ifdraft{\ \tt [#1]}\else\ignorespaces\fi}\begin{equation}\label{eq:#1}{#2}\end{equation}}

\newcommand {\equ}[1]     {\eqref{eq:#1}}

\newcommand\nz{\smallsetminus \{0\}}
\newcommand\Mtilde{T}

\newcommand{\ggm}{G/\Gamma}

\newcommand\hd{Hausdorff dimension}
\newcommand\ssm{\smallsetminus}

\newcommand\N{\mathbb{N}}
\newcommand\Q{\mathbb{Q}}
\newcommand\R{\mathbb{R}}
\newcommand\Z{\mathbb{Z}}

\begin{document}
\date{\today}
\title[Singular systems and non-escape of mass]{Singular systems of linear forms and non-escape of mass in the space of lattices}

\author{S.\@ Kadyrov}
\author{D.\@ Kleinbock}
\author{E.\@ Lindenstrauss}
\author{G.~A.\@ Margulis}
\thanks{
We gratefully acknowledge support 
by the National Science Foundation through grants DMS-1101320 (D.K.) and 
DMS-1265695 (G.M.), by
the EPSRC through grant EP/H000091/1 (S.K.), by the European Research Council through AdG Grant 267259 (E.L.), and by the ISF grant 983/09 (E.L.). D.K. and E.L. stay at MSRI was supported in part by NSF grant no. 0932078 000.}

\address[SK]{Mathematics Department,
Nazarbayev University,
Astana, Kazakhstan}
\email{shirali.kadyrov@nu.edu.kz}

\address[DK]{Department of Mathematics, Brandeis University, Waltham MA} 
\email{kleinboc@brandeis.edu}

\address[EL]{Einstein Institute of Mathematics, The Hebrew University of
Jerusalem, Jerusalem, Israel}
\email{elon@math.huji.ac.il}

\address[GM]{Department of Mathematics, Yale University, New Haven CT}
\email{grigorii.margulis@yale.edu}


\begin{abstract}

{Singular systems of linear forms were introduced by Khintchine in the 1920s, and it was shown by Dani in the 1980s that they are in one-to-one correspondence with certain divergent orbits of one-parameter diagonal groups on the space of lattices. We give a (conjecturally sharp) upper  bound on the \hd\ of the set of singular systems of linear forms (equivalently the set of lattices with divergent trajectories) as well as the dimension of the set of lattices with trajectories `escaping on average' (a notion weaker than divergence). This extends work by Cheung, as well as by Chevallier and Cheung. Our method differs considerably from that of Cheung and Chevallier, and is based on the technique of integral inequalities developed by  Eskin, Margulis and Mozes.}

\end{abstract}

\maketitle

\section{Introduction}

For any given $m,n \in \N$ we consider the space of unimodular $(m+n)$-lattices $X_{m+n} :=  \ggm$ where $G = \SL(m+n,\mathbb R)$ and $\Gamma =  \SL(m+n,\mathbb Z)$, and  the one-parameter diagonal semigroup $\{g_t\}_{t \ge 0}$, where 
\eq{defgt}{ g_t := \diag(e^{nt},...,e^{nt},e^{-mt},...,e^{-mt}).} 
We consider the action $g_t: X_{m+n} \to X_{m+n}$ by left translations, $g_t\cdot x=g_t x$.
The unstable horospherical subgroup  $U$ with respect to $g_1$ can be identified with the space $M_{m,n}$ of $m\times n$ real matrices:  \eq{uhs}{U = \{u_s : s \in M_{m,n}\}\text{ where  }u_s := \begin{pmatrix}I_m & s\\ 0 & I_n\end{pmatrix}\,.} 
The one-parameter diagonal semigroup $\{g_t\}_{t\geq 0}$ and the corresponding horospherical subgroup $U<G$ are closely connected to the diophantine properties of $m\times n$ real matrices.
One says that an $m \times n$ real matrix $s$ is a  {\sl singular system of $m$ linear forms in $n$ variables\/}  if for any $\varepsilon>0$ there exists $T_0 \in \N$ such that for any $T>T_0$ there exist ${\bf q} \in \Z^n$ and ${\bf p} \in \Z^m$ such that\footnote{Note that this definition is independent of the choice of norms on $\R^m$ and $\R^n$; later it will be convenient to work with the Euclidean norms. }
\eq{sing}{\|s {\bf q} -{\bf p}\| < \frac{\varepsilon}{T^{n/m}} \text{ and } 0<\|{\bf q}\|<T. }

Equivalently one can restrict  $T$ to be a power of a fixed natural number, e.g.\ only consider $T = 2^\ell$, $\ell\in \N$. This property was introduced by A.~Khintchine \cite{Kh} who showed that the set of such matrices has Lebesgue measure zero, hence the name `singular'. 
Later it was shown by S.G.\ Dani \cite{Dani}  that $s$ is  singular  if and only if the trajectory  $\{g_tu_s \Gamma: t \ge 0\}$
 is  divergent in $X_{m+n}$, that is, leaves every compact subset of $X_{m+n}$. Thus the zero measure of the set of singular systems  follows from the ergodicity of the $g_t$-action on $X_{m+n}$.

When $m = n = 1$, it is easy to see that each divergent trajectory $\{g_tx\}$, where $x\in X_2$ is a unimodular lattice in $\R^2$, is, in Dani's terminology, {\sl degenerate\/}, that is, there exists a subgroup of $x$ contracted by the action. This is a manifestation of the fact that the group $\SL(2,\mathbb R)$ has $\Q$-rank $1$, cf. \cite[Theorem~6.1]{Dani}. In particular, it follows that the set of singular real numbers (equivalently, $1\times 1$ matrices) coincides with $\Q$ and thus has \hd\ zero\footnote{This can also be easily shown using continued fractions.}.
However when $\max(m,n) > 1$ (that is, when $G$ has real rank bigger than $1$) it is possible to construct trajectories which diverge for non-trivial reasons. This was first observed by Khintchine and then generalized by Dani \cite[Theorem~7.3]{Dani}. 
Thus the set of points with divergent orbits has a quite complicated structure. In particular, computing its \hd\ is a difficult problem. When $m+n = 3$ it  was shown by Y.\ Cheung  \cite{Che11}  that the Hausdorff dimension of the set of singular pairs is $4/3$, hence the set of points in   $X_{3}$ with divergent $g_t$-orbits has \hd\ $22/3$, that is, codimension $2/3$. Also, recently in \cite{CC14} Cheung and N.\ Chevallier showed that the set of singular $m$-vectors has Hausdorff dimension $\frac{m^2}{m+1}$, which corresponds to codimension $\frac m{m+1}$ of the set of divergent trajectories in  $X_{m+1}$. 

Now let us say that a point $x \in X_{m+n}$ {\sl escapes on average} (with respect to the semigroup $g_t$ as in \equ{defgt} which we shall fix for the duration of the paper) if 
$$\lim_{N\to \infty} 
\frac{1}{N}\big|\big\{\ell \in \{1,\dots,N\}:g_\ell x \in Q\big\}\big|=0$$
for any compact set $Q$ in $X_{m+n}$. Observe that this notion is independent on the parametrization of the orbit, in other words, $x$ escapes on average if and only if 
$$\lim_{N\to \infty} 
\frac{1}{N}\big|\big\{\ell \in \{1,\dots,N\}:g_{a\ell} x \in Q\big\}\big|=0$$
for any $a > 0$ (or else one can replace summation by integration). 
In this paper we prove
\begin{thm}\label{thm:escape} For any $x\in X_{m+n}$, the set 
\eq{esconav}{
\{u\in U : ux \text{ escapes on average}\}} has Hausdorff dimension at most $mn- \frac{mn}{m+n}$.
Consequently, 
 \eq{on average}{\dim\left(\{x\in X_{m+n}: x \text{ escapes on average}\}\right) \le \dim (X_{m+n}) - \frac{mn}{m+n}\,.}  
\end{thm}

One can also consider a corresponding concept in Diophantine approximation, weakening the classical notion of singularity: say that  an $m \times n$ real matrix~$s$ is   {\sl singular on average\/}  if for any $\varepsilon>0$  one has
$$
\lim_{N{\to \infty}} %\frac{1}{k}\left|\left\{\ell \in {\{0,\dots,k-1\}}: 
\frac{1}{N}\left|\left\{\ell \in \{1,\dots,N\}:
\begin{aligned} \exists\, {\bf q} \in \mathbb Z^n \text{ and } {\bf p} \in \mathbb Z^m \text{ such that } \\\equ{sing} \text{ holds for } T = 2^\ell \qquad\end{aligned}\right\}\right|=1\,.
$$

Using Dani's correspondence, as an immediate corollary of Theorem~\ref{thm:escape} we obtain the following

\begin{cor}\label{cor:dim} The Hausdorff dimension of the set of $s\in M_{m,n}$ which are singular 
on average (and hence of the set of singular $s\in M_{m,n}$) is at most
$mn- \frac{mn}{m+n}$.
\end{cor}

Clearly, the set of points with divergent trajectories is contained in the set of points escaping on average. Thus the lower estimate from \cite{CC14} implies that  in the  case \eq{vectors}{\min(m,n) = 1,\ m+n \ge 3}  the bound in \equ{on average} is sharp; the same has been established in  \cite{KP12}  when $m=n=1$. We conjecture that the equality in \equ{on average} holds in all dimensions $m,n$. Also, it is natural to conjecture that, unless $m=n=1$, 
 the set of points with divergent trajectories has the same \hd\ as of those who escape on average; this  follows from  \cite{CC14} and Theorem \ref{thm:escape} in the case  \equ{vectors}.

Our second 
 result  relates the entropy of an invariant measure on $X_{m+n}$ to its mass in a compact set, in a way which is uniform over all invariant probability measures on $X_{m+n}$.

\begin{thm}\label{thm:entropy}
For every $\varepsilon>0$ there exists a compact subset $Q=Q(\varepsilon)$ of $X_{m+n}$ such that 
\eq{concli1.3}
{h_\mu(g_1)\le \big(m+n-1+\mu(Q)\big)mn+\varepsilon}
for any $g_1$-invariant probability measure $\mu$ on $X_{m+n}$.
\end{thm}

As a consequence of Theorem~\ref{thm:entropy} we obtain the following.

\begin{thm}\label{thm:sequence}
For any $h > 0$ and any sequence $(\mu_k)_{k\ge 1}$ of $g_1$-invariant probability measures on $X_{m+n}$ with entropies $h_{\mu_k}(g_1) \ge h$, any weak$^*$ limit $\mu$ of the sequence satisfies 
$$\mu(X_{m+n}) \ge \frac{h}{mn}-(m+n-1)\,.$$
\end{thm}

We remark that the maximal entropy is $(m+n)mn$, and for any value $h \in \big((m+n-1)mn, (m+n)mn\big]$  Theorem~\ref{thm:sequence} produces a nontrivial result. A similar statement was first proved in \cite{ELMV13} for the geodesic flow on the unit tangent bundle to the hyperbolic plane. Later, various generalizations were considered in \cite{EK12, Kad12(b),EKP13}. Theorem~\ref{thm:entropy} can be thought as a  generalization of   analogous results from \cite{ELMV13, EK12}. However the method used in the present paper is different from the previous work and crucially relies on the ideas from \cite{EMM}. {There are also interesting parallels to these results in the context of moduli spaces of abelian and quadratic differentials, in particular 
 \cite{Hamendstadt} by U.\ Hamenst\"adt; \cite{Athreya} and~\cite{Eskin-Mirzakhani} give applications of the \cite{EMM} techniques in the context of these moduli spaces that are also relevant.
}

We believe that Theorem~\ref{thm:sequence} is sharp in the sense that for any constant $h \in \big[0, (m+n)mn\big]$ there should exists a sequence of probability invariant measures $(\mu_k)_{k \ge 1}$ with $\lim_k h_{\mu_k}(g_1) =h$ such that the limit measure $\mu$ satisfies $\mu(X_{m+n})=\max\{\frac{h}{mn}-(m+n-1),0\}$. In \cite{Kad12(a)}, the claim was proved to be true when $\min(m,n)=1$.

Both Theorem~\ref{thm:escape} and Theorem~\ref{thm:entropy} are derived from 
 Theorem~\ref{thm:main}, the main result of this paper. In what follows we fix $m,n$ and  for the sake of brevity denote  $X_{m+n}$ by $X$. Take $U$ as in  \equ{uhs} and let $d_U$ be the distance induced by the Euclidean norm $\|\cdot\|$ on $M_{m,n}$ via the map $s\mapsto u_s$.  Also let 
\eq{ball}{B_r^U:=\{u_s :   \|s\| < r\} = \{u\in U : d_U(u,e) < r\}}
be the ball of radius $r$ centered at the identity element.
% where $\|\cdot\|$ is the Euclidean norm on $M_{m,n}$}. 
Then given a compact subset $Q$ of $X$, $N\in\N$, 
%$M> 0$, 
$\delta \in (0,1)$, $t>0,$ and $x \in X$, define the set 
\eq{zx}{Z_x(Q,N,t,\delta):=\left\{u \in B_{1}^U : \frac{1}{N}\big|\big\{\ell \in \{1,\dots,N\}:g_{t\ell} u x  
%\in X_{\ge M}
\,\notin Q\big\}\big| \ge \delta  \right\};}
in other words, the set of  $u\in B_{1}^U$ such that up to time $N$, the proportion of times $\ell$ for which the orbit point $g_{t\ell} ux$ is in the %cusp neighborhood 
%$ X_{\ge M}$
complement of $Q$  is at least~$\delta$.

The following statement is a covering result which we need for the proof of previously mentioned theorems:

\begin{thm}\label{thm:main}
There exists $t_0>0$ and a function $C: X \to \R_+$ such that the following holds. For any $t >t_0$ there exists  
%height function ${\tilde \alpha}$ and a real number $M_0>0$ 
a compact set $Q:=Q(t)$ of $X$ such that for any $N\in\N$, %$M\ge M_0$, 
$\delta \in (0,1)$, and $ x \in X$, the set
$Z_x(Q,N,t,\delta)$
can be covered with $ C(x)t^{3N} e^{(m+n-\delta)mntN }$ balls in $U$ of radius $e^{-(m+n)tN}$.
\end{thm}

\begin{remark}\label{defining alpha}
The function $C$ can actually be made precise as follows. 
For any $i=1,\dots,m+n$ and any $x \in X$ we let $F_i(x)$ denote the set of all $i$-dimensional subgroups 
of $x$ (we recall that the latter is viewed as a unimodular lattice in $\R^{m+n}$). For any $L \in F_i(x)$ we let $\|L\|$ denote the volume of $L/(L \cap x)$ with respect to the standard Euclidean structure on $\R^{m+n}$. 
(Equivalently,  $\|L\|=\|v_1  \wedge \cdots \wedge v_i\|$ where  $\{v_1, \dots,v_i\}$  is a basis for $L$.) Then, following \cite{EMM}, define 
\eq{defalphai}{\alpha_i(x):=\max\left\{\frac{1}{\|L\|}: L \in F_i(x)\right\}}
%Clearly, $\alpha_{m+n}(x)=1$, and for convenience we also set $\alpha_0(x)=1$ for all $x \in X$. 
and take
\eq{cx}{C(x):=\max\{\alpha_i^{\beta_i}(x):i=1,2,\dots,m+n-1\},}
where 
for any $i \in \{1,\dots,m+n-1\}$ we let \eq{defbeta}{\beta_i:=\begin{cases}\frac{m}{i}\qquad\text{ if }i \le m\\   \frac{n}{m+n-i}\text{ if }i>m.\end{cases}}
\end{remark}
%$\beta_i$'s and $\alpha_i$'s are as in \S\ref{sec:estimates}.

%$\alpha_i$'s and $\beta_i$'s as in 
In the next section, we show how to deduce Theorems ~\ref{thm:escape} and \ref{thm:entropy} from Theorem~\ref{thm:main}. From \S3, the rest of the paper is devoted to obtaining Theorem~\ref{thm:main}.

\subsection*{Acknowledgements}  We thank J.~Athreya, Y.~Cheung, M.~Einsiedler, A.~Eskin and M.~Mirzakhani for helpful discussions, and the referee for useful comments. We are also grateful to MSRI, where this paper was finalized, for its hospitality during Spring 2015. E.L. also thanks the Israel IAS for ideal working conditions during the fall of 2013.

\section{Proofs assuming Theorem~\ref{thm:main}}

\subsection*{Notation} 
  In what follows, by $x \ll y$ (resp., $x \asymp y$) we mean $x<Cy$ (resp., $c y <x<Cy$) for some absolute constants $c,C>0$ depending only on $m,n$. Fix a right-invariant Riemannian metric %$d_G$ 
  on $G$, inducing a metric on $X$ which will be denoted  by $d_X$. We let $B_r^G$ denote the open ball in $G$ of radius $r$ centered at the identity element.
  %, and when $g=1$ is the identity in $G$ we will simply write $B_r^H$ for $B_r^H(1).$

\begin{proof}[Proof of Theorem~\ref{thm:escape}] Let $x \in X$ be given. Fix $t_0>0$ as in 
Theorem~\ref{thm:main}, and for any $t>t_0$ choose the compact set $Q$ as in Theorem~\ref{thm:main}.
Then, using Theorem~\ref{thm:main} we get that for any $N \in \N$, each set $Z_x(Q,N,t,\delta)$ 
can be covered with $C(x)t^{3N} e^{(m+n-\delta)mntN}$ balls of radius $e^{-(m+n)tN}.$

Denote by $Z_x$ the set of all $u\in B_1^U$ such that $u x$ escapes on average. Note that
\begin{equation}\label{eqn:Zx}
Z_x \subset \bigcup_{N_0 \ge 1}\bigcap_{N \ge N_0} Z_x(Q,N,t,\delta) \text{ for any } \delta \in (0,1) %\text{ and } M>1
.
\end{equation}
Using Theorem~\ref{thm:main}, we can write
%Observe that a point $x$ in $X$ escapes on average if and only if for any $M\ge 1$ and $t>0$, 
%$$\lim_{N\to\infty} \frac{1}{N}\big|\big\{\ell \in \{1,\dots,N\}:g_{t \ell} x  \in X_{\ge M}\big\}\big|=1.$$
%We first show that for any $x \in X$ the set $Z_x$ of all $u\in B_1^U$ such that $u x$ escapes on average has upper box (and hence Hausdorff) dimension at most $\frac{(m+n-1)mn}{m+n}.$ 
%It follows that 
\begin{multline*}
\overline \dim_{\rm box} \Big(\bigcap_{N \ge N_0} Z_x(Q,N,t,\delta)\Big) \le \overline\lim_{N \to \infty} \frac{\log \big(C(x) t^{3N} e^{(m+n-\delta)mntN }\big) }{-\log ( e^{-(m+n)tN}) }\\
=\frac{3\log t + (m+n-\delta)mnt}{(m+n)t}.
\end{multline*}
This is true for any $t>t_0$ and any $\delta \in (0,1)$; thus, letting $t \to \infty$ and $\delta \to 1$, we get from \eqref{eqn:Zx} that
$\dim (Z_x) \le  \frac{(m+n-1)mn}{m+n}.$ Since the set \equ{esconav} is contained in a countable union of the sets of the form $Z_x$, the first part of Theorem~\ref{thm:escape} follows.

Now, let $P^0$ be the weak stable horospherical subgroup with respect to $g_1$, namely
\begin{equation}\label{eqn:weaks}P^0:= \left\{ \begin{pmatrix}s' & 0\\ s & s''\end{pmatrix}\left|\begin{aligned}\  s \in M_{n,m} , \ s'\in M_{m,m},\ s''\in M_{n,n}\\  \det(s')\det(s'') = 1\qquad\quad\end{aligned}\right.\right\}.\end{equation}

 Every element of a neighborhood of identity in $G$ can be written as $gu$ where $u$ belongs to a neighborhood of identity in $U$ and $g\in P^0$. Note that for $g \in P^0$ the union
$$
\bigcup_{t > 0} g_t g g_{-t}$$ is contained in a compact subset of $G$. Writing $g_tgu =  g_t g g_{-t}(g_tu)$, one sees that $ux$ escapes on average if and only if so does $gux$. Therefore the `consequently' part follows  from the slicing properties of the \hd.
\end{proof}

\begin{remark} One can generalize the definition of escape on average, saying that  a point $x \in X$  {\sl $\delta$-escapes on average}, where $0 < \delta \le 1$,  if 
$$\lim_{N\to \infty} \frac{1}{N}\big|\big\{\ell \in \{1,\dots,N\}: g_\ell x \notin Q\big\}\big|\ge \delta$$
for any compact $Q\subset X$. The previous definition corresponds to $\delta = 1$. Our proof of Theorem~\ref{thm:escape} actually establishes that 
$${\dim\left(\{x\in X: x \text{\ $\delta$-escapes on average}\}\right) \le \dim (X) - \frac{\delta mn}{m+n}\,.}  $$
It seems plausible to conjecture that the above bound is sharp for all $m,n$ and any $0 < \delta \le 1$.
\end{remark}

We now proceed with the proof of Theorem~\ref{thm:entropy}. Let us fix $t_0$ as in Theorem~\ref{thm:main} and take $t >t_0$. We fix sufficiently small $\eta>0$ such that the ball $B_\eta^G$ is an injective image under the exponential map of a neighborhood of $0$ in the Lie algebra of $G$. For any $N \in \N$ we define a {\sl Bowen $N$-ball\/} to be any set of the form $\Bow(N) x$ where $x \in X$ and
\eq{defbow}{\Bow(N):=\bigcap_{\ell=0}^{N-1} g_{-t\ell} B_{\eta}^{G}g_{t\ell}.}
We need the following lemma, which relates the entropy and covers by Bowen balls. It essentially is due to Brin and Katok \cite{BK83}, though there are some modifications needed for the non-compact case.

\begin{lem}
\label{lem:entropy}Let $\mu$ be an ergodic $ g_t$-invariant probability measure on $X$ and let $A\subset X$ be a measurable subset with $\mu(A) > 0$. For any $N\geq 1$  let $BC(A,N)$ be the minimal number of
Bowen $N$-balls needed to cover $A$. Then
$$
h_{\mu}( g_t)\le \liminf_{N \to \infty}\frac{\log BC(A,N)}{N}.
$$
\end{lem}

For a proof for $G=\SL(2,\mathbb R)$ see e.g.\ \cite[Lemma B.2]{ELMV13} and the remarks following it. The adaptation to $\SL(m+n,\mathbb R)$ is straightforward.

\begin{proof}[Proof of Theorem~\ref{thm:entropy}] Note first that it is sufficient to
consider ergodic measures. For if $\mu$ is not ergodic, we can
write $\mu$ as an integral of its ergodic components $\mu=\int
\mu_r d\nu(r)$ for some probability space $(E,\nu)$, see for example \cite{EW11}. Therefore, for any compact subset $Q$ of $X$ we have $\mu(Q)=\int_E \mu_r(Q)d\nu(r)$, but also $h_{\mu}( g_t)=\int_E h_{\mu_r}( g_t) d\nu(r)$, see for example \cite[Theorem~8.4]{Wal65}; hence the desired estimate follows from the ergodic case (note that we have used the fact that $Q$ does not depend on the measure at hand). From now on we assume $\mu$ to be ergodic.

Let $\varepsilon>0$ be given. Fix a sufficiently large $t>t _0$ such that $3\log t /t <\varepsilon$ 
with $t_0$ as in Theorem~\ref{thm:main}. For this $t$ let 
$Q$ be as in Theorem~\ref{thm:main}. We will establish the conclusion of 
Theorem~\ref{thm:entropy} for this compact set $Q$.

Note that we can assume that $\mu(Q) <1$, since otherwise  \equ{concli1.3} holds trivially due to the fact that $(m+n)mn$ is  an upper bound for $h_\mu(g_1)$. Choose $c > 0$ such that \eq{incl1}{B^U_{c\eta}\subset B^G_{\eta/2}} (we recall that the metric on $U$ does not coincide with the restriction of the metric on $G$, but locally near the identity element those metrics are close to each other). Also choose an open neighborhood $W$ of identity in the group  $P^0$ is as in \eqref{eqn:weaks} such that \eq{incl2}{B^U_{c\eta}W\subset B^G_{\eta},}
which is possible in view of \equ{incl1}.

Let us fix any $\mu$-generic point $x \in X$. Then 
$$\varepsilon':= \frac{\mu(B_1^U W x)}{2}>0.$$
The pointwise ergodic theorem implies
\eq{average}{
\frac{1}{N}\sum_{n=0}^{N-1}1_{
%X_{\ge M}
{X\ssm Q}}\big( g_{nt}(y)\big)\to \mu(
%X_{\ge M})
{X\ssm Q})
}
as $N \to \infty$ for $\mu$-a.e.\ $y \in X$. In particular, for any 
$\varepsilon'' \in \big(0,\mu(
%X_{\ge M}
{X\ssm Q})\big)$ there exists $N_0$ such that for $N>N_0$ the average in the
left hand side of \equ{average} will be bigger than $\mu(
{X\ssm Q}
%X_{\ge M}
)-\varepsilon''$ for any $y \in
Y$ for some $Y \subset X$ with measure $\mu(Y)> 1-\varepsilon'$.
Since  $\mu(B_1^U W x)=2\varepsilon'$, we see that
$\mu(Z)>\varepsilon'$ where 
$$Z:=Y\cap B_1^U Wx.$$ 
We now consider the covering of this set $Z$ by Bowen $N$-balls. 
Note that
$$B_{c\eta e^{-(m+n)tN}}^U W\subset \Bow(N)$$
in view of \equ{incl2} and \equ{defbow}.
Thus  it suffices to consider a covering of $$Z':=\{u \in B_1^U: ux \in Z\}=\{u \in B_1^U: ux \in Y\}$$ with balls of radius $c\eta e^{-(m+n)tN}$. Denote $\mu({X\ssm Q})-\varepsilon''$ by $\delta$; then  clearly
$$Z'\,\subset \bigcap_{N>N_0} Z_x(Q,N,t,\delta).$$
Applying Theorem~\ref{thm:main} %for $\delta=\mu(
%X_{\ge M})
%X\ssm Q_M)-\varepsilon'$ 
we get that for any $N 
%\in \N
> N_0$, the set $Z'$ can be covered with $C(x) t^{3N} e^{(m+n-\delta
%\mu(
%X_{<M}
%X\ssm Q_M)+\varepsilon'
)mntN }$ balls of radius $e^{-(m+n)tN}$. %Since 
Observe that a ball of radius $e^{-(m+n)tN}$ in $U$ can be covered by finitely many 
translates of a ball of radius $
%\frac{\eta}{2}
c\eta e^{-(m+n)tN}$. Thus for any $N %\in \N
\,> N_0$, the set $Z$ %of measure $\mu(Z)>\varepsilon$ 
can be covered with $\ll C(x) t^{3N} e^{(m+n-
%\mu(X_{\ge M})+\varepsilon'
\delta)mntN }$ %translates of 
Bowen $N$-balls. Since  $\mu(Z)>\varepsilon'>0$, from Lemma~\ref{lem:entropy} it follows that
 \begin{eqnarray*}
  h_{\mu}( g_t)&\leq
&\liminf_{N \rightarrow
\infty}\frac{\log BC(Z,N)}{N} \\
&\leq &  \big(m+n-\mu(
%X_{\ge M}
{X\ssm Q})+\varepsilon''\big)mnt+3\log t.
\end{eqnarray*}
Since $\varepsilon''\in \big(0,\mu(
%X_{\ge M}
{X\ssm Q})\big)$ is arbitrary and $3\log t/t < \varepsilon$, %replacing $\mu(X_{\ge M})$ by $1-\mu(X_{<M})$ 
we arrive at
\begin{equation}\label{eqn:entt}
h_\mu(g_1)=\frac{1}{t}h_\mu(g_t)\le \big(m+n-\mu\big(
%X_{<M}
 {X\ssm Q})\big)mn+\varepsilon.
\end{equation}
This finishes the proof.
%By considering $k_\mu+\frac{1}{\ell} \to k_\mu $ we see that \eqref{eqn:entt} also holds for $k_\mu$. 
%\bf If $\mu(Q_{k_\mu}) = 0$  we are done, but otherwise I do not know what to do...
%Also, since $\mu(X_{\le M})=0$ for any $M < M_\mu$ it follows that $\mu(X_{<M_\mu})=0$. Thus, for $M=M_\mu$ from \eqref{eqn:entt} we get
%$$\frac{1}{t}h_\mu(g_t)\le (m+n-1)mn+\frac{3\log t}{t}.
%$$
%Thus, \eqref{eqn:entt} trivially holds for $M =M_0$ as $\mu(X_{<M_0})=\mu(X_{<M_\mu})=0$.
 \end{proof}
 
 We end this section by giving the proof of Theorem~\ref{thm:sequence}.
 
 \begin{proof}[Proof of Theorem~\ref{thm:sequence}]
 Let us 
 %fix sufficiently large 
 take  $\varepsilon>0$ and compact subset $Q = Q(\varepsilon)$ as in Theorem~\ref{thm:entropy}. Since, $h_{\mu_k} \ge h$, using Theorem~\ref{thm:entropy} we see that for any $k \in N$
 $$\mu_k(Q) \ge \frac{h}{mn} - (m+n) +1 -\frac{\varepsilon}{mn}.$$
% We note that we cannot immediately take $t \to \infty$ and omit the last term as $Q$ depends on $t$.
 Pick a compactly supported continuous function $f:X \to [0,1]$ such that $f(x)=1$ on $Q$. Then, $\int f d\mu_k \ge \mu_k(Q)$. Let $\mu$ be a weak$^*$ limit of $(\mu_k)_{\ge 1}$. Then, letting $ k\to \infty $ we see that
 $$\mu(X) \ge \int f d \mu \ge  \frac{h}{mn} - (m+n) +1 -\frac{\varepsilon}{mn}.$$
 To finish the proof we now let $\varepsilon\to 0$.
 \end{proof}
 
 For the rest of the paper our goal is to prove Theorem~\ref{thm:main}.

\section{Estimates for certain integral operators}\label{sec:estimates}

We fix the standard Euclidean structure on $\R^{m+n}$, let $\{e_1,\dots,e_{m+n}\}$ be the standard basis of $\R^{m+n}$ and let $K = \SO(m+n)$ be the group of orietnation-preserving linear isometries of $\R^{m+n}$ (maximal compact in $G$). The main goal of this section is to prove an estimate for averages of certain functions over $K$. We let $dk$ stand for the normalized Haar measure on $K$.

\begin{prop}\label{prop:gaussian}
For any $i \in \{1,\dots,m+n-1\}$ we let $\beta_i=\frac{m}{i}$ if $i \le m$ and $\beta_i=\frac{n}{m+n-i}$ if $i>m$, as defined in \equ{defbeta}. Then there exists $c>0$ (dependent only on $m,n$) such that
$$\int_K \|g_tkv\|^{-\beta_i} \,dk \le c t e^{-mnt}\|v\|^{-\beta_i},$$
for any $t\ge 1$ and any decomposable $v \in \bigwedge^i \R^{m+n}.$
\end{prop}
We note that except for {$i=1$, $m$ and $m+n-1$} the factor $t$ in the right hand side is not necessary. 

We need some lemmas. 

\begin{lem}\label{lem:lessinfty}
Let $i\in \{1,\dots,d\}$ and let $x_1,\dots,x_i$ be independent $K$-invariant standard Gaussian random variables on $\R^d$. Then
$$\mathbb E(\|x_1\wedge \cdots \wedge x_i\|^{-\beta})< \infty$$
for $\beta<d-i+1$.
\end{lem}

We will be applying the lemma both for $d=m$ and $d=m+n$.

\begin{proof}
Let $\pi_{\perp,j}$ denote the orthogonal projection to the space perpendicular to $x_1,\dots,x_j$. Then,
for $\ell \in \N$  
\begin{equation*}
{\rm Prob} (\|x_1\wedge \cdots \wedge x_i\|  \asymp e^{-\ell})\asymp \sum_{0\le \ell_1,\ell_2,\dots,\ell_i \le \ell} \prod_{j=1}^i {\rm Prob}(\|\pi_{\perp,j-1} x_j\| \asymp e^{-\ell_j}), 
\end{equation*}
where $\ell_1,\dots,\ell_i$ run through integers with $\sum \ell_j =\ell$. Now, $$ {\rm Prob}\big(\|\pi_{\perp,j-1} x_j\| \asymp e^{-\ell}\big)$$ is the probability a standard Gaussian in $\R^{d-j+1}$ has size $\asymp e^{-\ell}$, which is $e^{-(d-j+1)\ell}$. It is now easy to conclude that 
\begin{equation}\label{eqn:probell}
{\rm Prob} (\|x_1\wedge \cdots \wedge x_i\|  \asymp e^{-\ell}) \asymp e^{-(d-i+1)\ell},
\end{equation}
and the lemma follows.
\end{proof}

\begin{cor}\label{cor:prob}
Let $i\in\{1,\dots,d\}$ and let $A$ be some event depending on $x_1,\dots,x_i$, with $x_1,\dots,x_i$ standard Gaussians as before. Take  $\beta<d-i+1$. Then
$$\mathbb E(\|x_1\wedge \cdots \wedge x_i\|^{-\beta} \mid A) \ll {\rm Prob}(A)^{-\frac{\beta}{d-i+1}}.$$
\end{cor}

\begin{proof} 
Observe that from all events $A$ with given ${\rm Prob}(A)=p$,  
$$\mathbb E(\|x_1\wedge \cdots \wedge x_i\|^{-\beta} \mid A)$$
is maximal for $A$ of the form $$A=\{(x_1,\dots,x_i) : \|x_1\wedge \cdots \wedge x_i\| \le \sigma\},$$ where $\sigma$ is chosen so that ${\rm Prob}(A)=p$. Write
\eq{aell}{A_\ell=\{x_1,\dots,x_i : \|x_1\wedge \cdots \wedge x_i\| \asymp \sigma e^{-\ell}\}\,;}
then $A=\cup_{\ell=0}^\infty A_\ell$ for an appropriate (uniform in $\ell$) choice of implicit constants in \equ{aell}.
Thus, using \eqref{eqn:probell} we get that
\begin{align*}
\mathbb E(\|x_1\wedge \cdots \wedge x_i\|^{-\beta} \mid A) 
&\ll \sum_{\ell=0}^\infty \mathbb E(\|x_1\wedge \cdots \wedge x_i\|^{-\beta} \mid A_\ell) \frac{\mathrm{Prob}(A_\ell)}{p}\\
&\ll p^{-1}\sum_{\ell=0}^\infty (\sigma e^{-\ell})^{-\beta} (\sigma e^{-\ell})^{d-i+1}  \ll p^{-1}\sigma^{d-i+1-\beta}.
\end{align*}
Using \eqref{eqn:probell} one more time we see that $p = \mathrm{Prob}(A) \asymp \sigma^{{d-i+1}}$ so that
$$\mathbb E(\|x_1\wedge \cdots \wedge x_i\|^{-\beta} \mid A) \le {\rm Prob}(A)^{-\frac{\beta}{d-i+1}}. $$
\end{proof}

For later purposes we also need the following weaker estimate for $\beta=d-i+1$. In this case we have

\begin{lem}\label{lem:im}
For $i\in \{1,\dots,d\}$ and $\kappa>0$, we have that
$$\mathbb E\left(\|x_1\wedge \cdots \wedge x_i\|^{-(d-i+1)}  \mathds{1}(\|x_1\wedge \cdots \wedge x_i\| > e^{-\kappa})\right) \ll \kappa,$$
where the implicit constant is independent of $\kappa$.
\end{lem}

\begin{proof}
Let $A_\ell$ be as in \equ{aell}.
Again using the estimate~\eqref{eqn:probell} we see that
\begin{align*}
\mathbb E\bigl(\| x_1 \wedge \cdots \wedge x_i \|^{-(d-i+1)} & \mathds{1}(\| x_1 \wedge \cdots \wedge x_i \| > e^{-\kappa})\bigr) 
\asymp \\ \asymp& \sum_ {\ell = 0} ^ \kappa \mathbb E(\| x_1 \wedge \cdots \wedge x_i \|^{-(d-i+1)} \mid A_\ell ) \mathrm{Prob}(A_\ell)\\
\asymp &\sum_ {\ell = 0} ^ \kappa 1 = \kappa
.\end{align*}
\end{proof}

\begin{proof}[Proof of Proposition~\ref{prop:gaussian}]
We assume that $i \le m$, and will deal with the other case by duality.

Consider
$$I_t(v)=\int_K \|g_t k v\|^{-\beta_i} \,dk,$$
for some decomposable $v=x_1\wedge \cdots \wedge x_i$. Since $K$ acts transitively on the 
%two connected components of the 
variety of the decomposable wedges 
%of a given norm (and these two components are swapped by the map $v \to -v$) 
up to homothety, we have 
$$I_t(v)=C(t) \|v\|^{-\beta_i}$$
for some function $C(t)$.
It follows that if $v=x_1\wedge \cdots \wedge x_i$ with $x_i$ independent ($K$-invariant) standard Gaussians and $\beta_i<m+n-i+1$ (which is certainly true in our choice of $\beta_i$) then
$$C(t)=\frac1C \mathbb E I_t(x_1\wedge \cdots \wedge x_i)=\frac1C \mathbb E(\|g_t (x_1 \wedge \cdots \wedge x_i)\|^{-\beta_i}),$$
where the last equality follows from $K$-invariance and 
$$C=\mathbb E(\|x_1\wedge \cdots \wedge x_i\|^{-\beta_i})$$
is independent of $t$. Thus, to prove the proposition we only need to show that if $x_1,\dots,x_i$ are standard Gaussian random variables then
\begin{equation}\label{eqn:expmnt}
\mathbb E(\|g_t (x_1\wedge \cdots \wedge x_i)\|^{-\beta_i}) \le c t e^{-mnt}\,,
\end{equation}
where $c$ is independent of $t$.

Let $V_u \subset \R^{m+n}$ denote the $m$-dimensional subspace spanned by $e_1,\dots,e_m$, and let $V_s$ be the complementary subspace, so that
$$\|g_t v\|=e^{nt} \|v\| \text{ and } \|g_t w\|=e^{-mt}\|w\|$$
whenever $v \in V_u$ and $w \in V_s$. In particular, for any $v \in \bigwedge^i V_u$ we have $\|g_t v\|=e^{int} \|v\|.$ Let $\pi^{(i)}_u:\bigwedge^i\R^{m+n} \to \bigwedge^i V_u$ be the natural (orthogonal) projection.

Clearly,
$$\pi_{u}^{(i)} (x_1\wedge \cdots \wedge x_i)=\pi_{u}^{(1)}(x_1)\wedge \cdots \wedge \pi_{u}^{(1)}(x_i),$$
and each of $\pi_{u}^{(1)}(x_j)$ is a standard Gaussian random variable in $m$ dimensions. 

To show \eqref{eqn:expmnt} we first assume that $1<i<m$. Clearly, one has 
$$\|g_t(x_1\wedge \cdots \wedge x_i)\| \ge  \|\pi_u^{(i)}g_t (x_1\wedge \cdots \wedge x_i)\|.$$
Then, using Lemma~\ref{lem:lessinfty} with $d=m$ we in fact get
\begin{multline*}
\mathbb E(\|g_t (x_1\wedge \cdots \wedge x_i)\|^{-\beta_i}) \le \mathbb E(\|\pi_u^{(i)} g_t (x_1\wedge \cdots \wedge x_i)\|^{-\beta_i} )\\=e^{-in\beta_i t} \mathbb E(\|\pi_u^{(i)}  (x_1\wedge \cdots \wedge x_i)\|^{-\beta_i} )\ll e^{-mnt},
\end{multline*}
 as $\beta_i=\frac{m}{i}<m-i+1$ for $1<i<m.$

Now, assume that $i=1$ or $i=m$. In this case we need to be a bit subtler, since $\beta_i=m-i+1$. However, we note that
$$\|g_t(x_1\wedge \cdots \wedge x_i)\| \ge \max \left(\|g_t\pi_u^{(i)}(x_1\wedge \cdots \wedge x_i)\|, e^{-imt} \|x_1\wedge \cdots \wedge x_i\|\right).$$
Hence, for any $\kappa>0$ 
$$\mathbb E(\|g_t (x_1\wedge \cdots \wedge x_i)\|^{-\beta_i}) \le E_1+E_2\,,$$
where 
$$E_1=\mathbb E\left(\|g_t\pi_u^{(i)} (x_1\wedge \cdots \wedge x_i)\|^{-\beta_i}{\mathds 1}(\|\pi_u^{(i)}(x_1\wedge \cdots \wedge x_i)\|>e^{-\kappa t})\right),$$
and
$$E_2=e^{im\beta_it}\mathbb E\left(\|x_1\wedge \cdots \wedge x_i\|^{-\beta_i}\mathds{1}(\|\pi_u^{(i)}(x_1\wedge \cdots \wedge x_i)\|<e^{-\kappa t})\right).$$
From Lemma~\ref{lem:im} for $d=m$ we get
\begin{align*}
E_1&=e^{-in\beta_i t}\mathbb E\left(\|\pi_u^{(i)} (x_1\wedge \cdots \wedge x_i)\|^{-\beta_i}\mathds 1(\|\pi_u^{(i)}(x_1\wedge \cdots \wedge x_i)\|>e^{-\kappa t})\right)\\
 &\ll \kappa t e^{-in\beta_i t}= \kappa t e^{-mnt}.
\end{align*}

To estimate $E_2$, write $$A=\{(x_1,\dots,x_i) : \|\pi_u^{(i)} (x_1\wedge \cdots \wedge x_i)\| < e^{-\kappa t}\}$$ and recall  that ${\rm Prob}(A) \asymp e^{-\kappa t (m-i+1)}.$ Hence, using Corollary~\ref{cor:prob} for $d=m+n$ we conclude
$$
E_2=e^{im\beta_it} {\rm Prob}(A) \mathbb E(\|x_1\wedge \cdots \wedge x_i\|^{-\beta_i} \mid A )
\le  \begin{cases} e^{m^2t} e^{-\kappa mt} e^{\frac{\kappa m^2 t}{m+n}} & \textrm{ if $i=1,$} \\ e^{m^2 t}e^{-\kappa t}e^{\frac{\kappa t}{n+1}}
& \textrm{ if $i=m$}.\end{cases}
$$
Either way, if $\kappa$ is a sufficiently large constant (depending on $n$ and $m$) then $E_2 \ll e^{-mnt}.$

This concludes the proof of the proposition for $i \le m$. 

For $i>m$ we exploit duality. Define the linear map
$$*:\bigwedge^j \R^{m+n} \to \bigwedge^{m+n-j} \R^{m+n} \text{ by } *(\wedge_{i \in I} e_i)=\wedge_{i \not\in I} e_i.$$
Then, 
$$*(k v)=k (*v), *(g_t v)=g_{-t} (*v), \text{ and } \|*v\|=\|v\|. $$
Therefore,
$$\int_U \|g_t k v\|^{-\beta} \,dk=\int_{K} \|*g_t k v\|^{-\beta} \,dk=\int_{K}  \|g_{-t} k (*v)\|^{-\beta}\,dk.$$
and the desired estimate follows by applying the above for $j'=m+n-j$,  $n'=m$ and $m'=n$.
\end{proof}

We now show that Proposition \ref{prop:gaussian} will remain valid if integration over $K$ is replaced with integration over a bounded subset of $U$ (with the constant dependent on that subset). %In fact, it will be convenient to integrate with respect to $d\rho$ which will denote the  Gaussian probability measure  on $M_{m,n}$ where each component is i.i.d.\ with mean $0$ and variance $1$.

\begin{lem}\label{lem:KUasymp} There exists a neighborhood $V$ of identity in $M_{m,n}$ such that for any $s_0 \in M_{m,n}$, $t,\beta>0$, $i \in \{1,\cdots,m+n-1\}$, and decomposable $w \in \bigwedge^i \R^{m+n}$ we have
$$\int_{V+s_0} \| g_t u_s w \|^{-\beta}\,ds \ll (1+\| s_0 \|)^{\beta} \int_K \| g_t k w \|^{-\beta} \,dk$$
with the implied constant dependent only on $m$, $n$ and $\beta$.
\end{lem}

\begin{proof}

We shall make use of the %following 
groups:
\begin{align*}
N = \left\{ \begin{pmatrix} 1 & * & \dots &*\\
0&1&\dots&*\\
\vdots& &\ddots & \vdots \\
0 & 0 & \dots & 1\end{pmatrix} \right\}  %<,U \\
%P &= \left\{ h \in G: \sup_ {t>0}\|g_t h g_{-t}\| < \infty \right\} \\
\quad \text{and}\quad N^0=N \cap P^0,
\end{align*}
where $P^0$ is as in (\ref{eqn:weaks}). Note that $ N=N^0U$ and $U \lhd N$.

There is a local diffeomorphism $f: K \to N $ such that in a neighborhood of identity $\mathcal{O}$ in $K$ we may write $k=p(k) f(k)$ for some $p(k) \in P^0$, with the Jacobian of $f$ bounded from above and below in this neighborhood.
Suppose $\mathcal{O} _ U, \mathcal{O} _ {N^0}$ are neighborhood of identity in $U$ and $N ^ 0$ respectively so that $ \mathcal{O} _ {N^0}\mathcal{O} _ U \subset f (\mathcal{O})$.

Then
\begin{align*}
\int_K \| g_t k w \|^{-\beta} \,dk &
\ge \int_{\mathcal{O}} \| g_t p(k)f(k)w \|^{-\beta} \,dk\\
& \ge \int_{\mathcal{O}} \left\|g_{t} p(k)^{-1} g_{-t}\right\|^{-\beta} \| g_t f(k) w\|^{-\beta} \,dk\\
& \gg \int_{\mathcal{O}_{N^0}}\int_{\mathcal{O}_U} \|g_t n_0 u w \|^{-\beta} \,dn_0du \\
&= \int_{\mathcal{O}_{N^0}}\int_{\mathcal{O}_U} \| n_0 g_t u w \|^{-\beta} \,dn_0du \\
&\gg \int_{\mathcal{O}_U} \|g_t u w \|^{-\beta} \,du
\end{align*}

We may assume that $\| w \|=1.$ For any $u \in U$ we may find $k \in K$ such that $kuw=\pm \| uw \| w$. Also, for any $s \in M_{m,n}$, 
\[\| u_s w \| \ge \| u_s^{-1}\|^{-1}\ge(1+\| s \|)^{-1}.\]
Setting $V=\{s \in M_{m,n}: u_s \in \mathcal{O}_U\}$ we have that for any $s_0 \in M_{m,n}$
\begin{multline*}
\int_{V+s_0}\| g_t u_s w \|^{-\beta} ds=\int_{V}\| g_t u_{s} u_{s_0} w \|^{-\beta} ds \ll \int_K \| g_t k (u_{s_0} w)\|^{-\beta}\,dk \\
=\| u_{s_0}w \|^{-\beta} \int_K \| g_t k w \|^{-\beta}\,dk \le (1+\| s_0 \|)^{\beta} \int_K \| g_t k w \|^{-\beta} \,dk.
\end{multline*}
\end{proof}

Recall that in Remark  \ref{defining alpha} we defined $\alpha_i(x)$ to be the maximum value of $1/\|L\|$  where $\|L\|$ is the volume of $L/(L \cap x)$ and $L$ runs through the set  $F_i(x)$  of all $i$-dimensional subgroups 
of $x$. Clearly, $\alpha_{m+n}(x)=1$, and for convenience we also set $\alpha_0(x)=1$ for all $x \in X$. 

 In the next corollary we  replace the integration over a neighborhood of identity in $U$ with the  integration over all of $U$ with respect to $d\rho$, which will denote the  Gaussian probability measure  on $M_{m,n}$ where each component is i.i.d.\ with mean $0$ and variance $1$.
Using Proposition~\ref{prop:gaussian} and Lemma~\ref{lem:KUasymp} we %now 
argue as in \cite[Lemma 5.7]{EMM} to obtain the following corollary:

\begin{cor}\label{lem:main} Let $\{\beta_i: i = 1,\dots,m+n-1\}$ be as in 
Proposition~\ref{prop:gaussian}. Then there exists $c_0 > 0 $ with the following property: given any 
 $t\ge 1$ one can choose $\omega >0$ such that for any $x\in X$ and $i \in \{1,\dots,m+n-1\}$ one has
\begin{multline*}
\int_{U}\alpha_i(g_t u_s x)^{\beta_i} \,d\rho(s) \le c_0 t e^{-mnt} \alpha_i(x)^{\beta_i}\\+\omega^{2\beta_i} \max_{0 < j \le \min\{m+n-i,i\}} \left(\sqrt{\alpha_{i+j}(x) \alpha_{i-j}(x)}\right)^{\beta_i}\,.
\end{multline*}
%where , and  $c$ is  independent of $t$ and $x$.
\end{cor}

\begin{proof}
For a given $x \in X$, let $L_i \in F_i(x)$ be a subgroup of $x$ such that
\begin{equation}\label{eqn:Li}
\alpha_i(x)=\frac{1}{\|L_i\|}.
\end{equation}
For any $i \in \{1,\dots,m+n\}$, $L \in F_i(x)$, and $s \in M_{m,n}$ one has
$$(1+\|s\|)^{-1} \|L\| \le \|u_{s} L\| \le (1+\|s\|) \|L\|.$$
Let $\omega=  \max_{0<j<m+n}\|\bigwedge^jg_t\| $; 
then
\begin{equation}\label{eqn:omega}
\omega^{-1} \le \frac{\|g_t v\|}{\|v\|} \le \omega \text{ for any } 0<j< m+n, \ v \in \textstyle \bigwedge^j( \R^{m+n})\nz.
\end{equation}
Let us consider
$$\Psi_i:=\{L : L \in F_i(x),\ \|L\|< \omega^2 \|L_i\|\}.$$
For any $L \in F_i(x) \smallsetminus \Psi_i$ we have 
$$\|u_s L_i\| \le (1+\|s\|) \|L_i\| \le \frac{1+\|s\|}{\omega^2} \|L\| \le \frac{(1+\|s\|)^2}{\omega^2} \|u_s L\| \text{ for any } s \in M_{m,n}.$$
Hence, for any $s \in M_{m,n}$ we see from \eqref{eqn:omega} that 
\begin{equation}\label{eqn:uniqueLi}
\|g_t u_s L_i\| \le (1+\|s\|)^2 \|g_t u_s L\|.
\end{equation}
First, assume that $\Psi_i=\{L_i\}$. Then, \eqref{eqn:uniqueLi} gives
\begin{equation}\label{eqn:1pls}
\int_{U} \alpha_i(g_t u_s x)^{\beta_i} \,d\rho(s) \le  \int_{U}(1+\|s\|)^{2\beta_i} \|g_t u_s L_i\|^{-\beta_i} \,d\rho(s).
\end{equation}
Clearly, for any $s_0 \in M_{m,n}$
\begin{multline*}
\int_{V+s_0}(1+\|s\|)^{2\beta_i}\| g_t u_s L_i \|^{-{\beta_i}} \,d \rho(s) \\
\ll\left( \max_{s \in V+{s_0}} (1+\|s\|)^{2\beta_i} e^{-\frac{\| s\|^2}{2}}\right)\int_{V+s_0} \| g_t u_s w \|^{-{\beta_i}}ds\\
\le e^{-\frac{\| s_0\|^2}{2}+O(\|s_0\|)}\int_{V+s_0} \| g_t u_s w \|^{-{\beta_i}}ds,
\end{multline*}
where the implied constants are independent of $s_0$. Summing over a lattice $\Lambda$ in the vector space $M_{n,m}$, sufficiently fine so that $M_{n,m}=V+\Lambda$, and using Lemma~\ref{lem:KUasymp} we see that 
\begin{align*}
\int_U (1+\|s\|)^{2\beta_i}\| g_t u_s L_i \|^{-\beta_i}\,d \rho(s) &\le \sum_{s' \in \Lambda} \int_{V +{s'}} (1+\|s\|)^{2\beta_i}\| g_t u_s L_i \|^{-\beta_i}\,d \rho(s) \\
&\ll \sum_{{s'} \in \Lambda} e^{-\frac{\| s'\|^2}{2}+O(\|s'\|)} \int_{V+{s'}} \| g_t u_s L_i \|^{-{\beta_i}}ds\\
&\ll \sum_{{s'} \in \Lambda}  (1+\| {s'}\|)^{{\beta_i}}e^{-\frac{\| s'\|^2}{2}+O(\|s'\|)} \int_{K} \| g_t k L_i \|^{-{\beta_i}}\,dk\\
& \ll \int_{K} \| g_t k L_i \|^{-{\beta_i}}\,dk.
\end{align*}
Thus, Proposition~\ref{prop:gaussian} and \eqref{eqn:1pls} give
$$\int_{U} \alpha_i(g_t u_s x)^{\beta_i} \,d\rho(s) \le c_0 t e^{-mnt}\|L_i\|^{-\beta_i}=c_0 t e^{-mnt} \alpha_i(x)^{\beta_i}$$
when $\Psi_i=\{L_i\}$, where $c_0$ depends only on $m,n$.
%Now, if $\Psi_i\neq\{L_i\},$ 
Otherwise let $L' \in \Psi_i$, $L' \neq L_i$. Then $\dim(L_i+L')=i+j$ for some $j>0$. From \eqref{eqn:Li}, \eqref{eqn:omega}  and \cite[Lemma~5.6]{EMM} we get for all $s \in M_{m,n}$ that
\begin{multline*}
\alpha_i(g_t u_s x)<(1+\|s\|)^{\beta_i}\omega \alpha_i(x)=\frac{(1+\|s\|)^{\beta_i}\omega}{\|L_i\|}<\frac{(1+\|s\|)^{\beta_i}\omega^2}{\sqrt{\|L_i\|\|L'\|}} \\\le \frac{(1+\|s\|)^{\beta_i}\omega^2}{\sqrt{\|L_i \cap L'\|\|L_i+L'\|}} \le (1+\|s\|)^{\beta_i}\omega^2 \sqrt{\alpha_{i+j}(x) \alpha_{i-j}(x)}.
\end{multline*}
Hence, if $\Psi_i \ne \{L_i\}$ then
\begin{multline*}
\int_{U} \alpha_i(g_t u_s x)^{\beta_i}\,d\rho(s) \\
\le \omega^{2\beta_i} \max_{0<j\le \max\{m+n-i,i\}}\left( \sqrt{\alpha_{i+j}(x) \alpha_{i-j}(x)}\right)^{\beta_i} \int_U (1+\|s\|)^{\beta_i} d\rho(s).
\end{multline*}
Since $\int_U (1+\|s\|)^{\beta_i} d\rho(s) \ll 1$, combining the above two cases leads to the desired result.
\end{proof}

\section{Choosing an appropriate height function}

In this section we will 
study  an abstract setting which will allow us to choose 
certain functions on $X$,   
to be used
later for constructing the compact subset $Q$ of Theorem~\ref{thm:main}.

\begin{prop}\label{prop:Delta}
Let $d \in \N$ be given, and let $\beta:\{0,\dots,d\} \to \R_+$ be a concave\footnote{More precisely, if the piecewise linear interpolation of $i\mapsto\beta(i)$ is a concave function $[0,d]\to\R_+$.} function such that $\beta(0)=\beta(d)=0$. Let $H$ be a set and let $A$ be a linear operator in the space of real functions on $H$ with $A(1)=1$. Suppose we are given functions $f_i:H \to \R_+$, where $ i=0,\dots, d$, such that  $ f_0=f_{d}=1$ and the following inequalities hold:
\begin{equation}\label{eqn:averaging}
A(f_i^{\beta_i}) \le {a} f_i^{\beta_i}+ \omega^{2\beta_i} \max_{0<j\le \min\{d-i,i\}}\left( \sqrt{f_{i+j}f_{i-j}}\right)^{\beta_i},  \  i=1, \dots,d-1,
\end{equation}
where $\beta_i=1/\beta(i)$ and ${a}$, $\omega$ are some positive constants. Then for any ${a}'>{a}$ there exist constants $\omega_0,\dots,\omega_{d}>0$ and $C_{0}>1$ such that the linear combination
\begin{equation}\label{eqn:f}{f:=\sum_{i=0}^{d} (\omega_i f_i)^{\beta_i}}\end{equation}
satisfies
{$$(A f) (h) \le {a}'f(h)+C_{0}\text{ for all }h\in H.$$}
\end{prop}

We note that the special case of the proposition when $$\beta(1)=\beta(2)=\cdots=\beta(d-1)$$ appears in \cite{EMM}. In our context  $H = X$,  $f_i = \alpha_i$ as in \equ{defalphai}, $d = m+n$,  $A$ is defined by $$(Af)(x) :=  \int_{U}f(g_t u_s x) \,d\rho(s)$$ for a fixed $t\ge 1$, and 
the function $\beta$ is given by $\beta(i)=1/\beta_i$, where $\beta_i$'s are as in 
%Proposition~\ref{prop:gaussian}
\equ{defbeta}. Clearly $\beta$ satisfies the convexity assumption. Take $c_0$ as in Corollary~\ref{lem:main}, and then, for arbitrary  $t \ge 1$, choose $\omega$ as in Corollary~\ref{lem:main} and let $a = c_0  t e^{-mnt}$ and $a' = 2c_0  t e^{-mnt}$. After that  take  $\omega_0,\dots,\omega_{m+n}$ and $C_{0}$ as in Proposition~\ref{prop:Delta}. Now consider  the linear combination
\begin{equation}\label{eqn:alphatilde}{\tilde \alpha :=\sum_{i=0}^{m+n} (\omega_i \alpha_i)^{\beta_i}};\end{equation}
%note that it is a height function. 
Proposition~\ref{prop:Delta} then implies that for any $x\in X$ one has
$$\int_{U} {\tilde \alpha}(g_t u_s x) \,d\rho(s) \le 2c_0 t e^{-mnt} {\tilde \alpha}( x) +C_{0}.$$  Furthermore, the right hand side of the above inequality  is not greater than $3c_0 t e^{-mnt} {\tilde \alpha}( x) $   if $\tilde \alpha( x) \, > \, \frac{C_{0}e^{mnt}}{c_0t}$.  This way one arrives at
%From the above statement one immediately deduces

\begin{cor}\label{cor:main} There exists a constant $c > 0$ with the following property: for any $t \ge 1$ one can choose  constants $\omega_0,\dots,\omega_{m+n} $  and $\Mtilde$   
such that for any $x \in X$ with ${\tilde \alpha}(x)\, > \,\Mtilde$, where ${\tilde \alpha}$ is as in \eqref{eqn:alphatilde},  one has
%for any $x \in X$ with ${\tilde \alpha}(x) \ge M
%$ one has
$$\int_{U} {\tilde \alpha}(g_t u_s x) \,d\rho(s) \le c  t e^{-mnt} {\tilde \alpha}(x)\,.$$
%where $\tilde \alpha$ is as in \eqref{eqn:alphatilde}, and the implicit constant is independent of $t$ and $x$.
\end{cor}

\begin{proof}[Proof of Proposition~\ref{prop:Delta}]
Let
%\begin{align*}
$$\Psi := \left\{ (i, j) \in \N^2: 0 < i < d, \quad 0<j \leq \min (i, d - i) \right\} $$
and
$$
\Phi = \left\{ (i, j) \in \Psi: \beta (i) = \tfrac12 \big(\beta (i - j) +  \beta (i + j)\big) \right\},
$$
%.\end{align*}
and define
\[
{b} :=   \max_ {(i,j) \in \Psi \smallsetminus \Phi} \left (\frac{\beta (i - j) +  \beta (i + j)}{2\beta(i)}\right)
\]
(note that ${b} < 1$ in view of the definition of $\Phi$). 
Let $\varepsilon\in (0,1)$ be a small parameter to be set later. Let $I$ be the set of indices $i \in \{0,\dots,d\}$ at which $\beta (\cdot)$ is strictly concave, i.e.\  such that $(i,j) \not\in \Phi$ for all $j$. By definition $0,d \in I$. Define for every $i \in \left\{ 0, \dots, d \right\}$
\[
d_-(i) := \min_{i' \leq i, i' \in I} (i-i'), \qquad d_+ (i) := \min _ {i ' \geq i, i' \in I} (i ' - i)\,;
\]
note in particular that $d_-(0)=d^+(d)=0$. 
%We take $\varepsilon $ small to be determined later, and f
Now for every $i \in \left\{ 0, \dots, d \right\}$ define 
\[ \omega _ i := \varepsilon ^ {d_-(i)d_+(i)}.\]
Clearly for any $i \in \left\{ 0, \dots, d \right\}$ we have   $f_i \leq \omega _ i ^{-1} f^{\beta (i)}$, where $f$ is defined as in  \eqref{eqn:f}. 
Applying \eqref{eqn:averaging}, we can write
\begin{align*}
Af &= \sum_{i=0}^d \omega _ {i} ^ {\beta _ i} A (f _ i ^ {\beta _ i}) 
=2 + \sum_{i=1}^{d-1} \omega _ {i} ^ {\beta _ i} A (f _ i ^ {\beta _ i})\\
& \leq 2 + \sum_{i=1}^{d-1} \omega _ {i} ^ {\beta _ i} \left( {a} f_i^{\beta_i}+ \omega^{2\beta_i} \max_{0<j\le \min\{d-i,i\}}\left( \sqrt{f_{i+j}f_{i-j}}\right)^{\beta_i}\right)\\
% ^ {\beta _ i} + \sum_ {i = 1} ^ {d-1} \omega^{2\beta_i} \max_{0<j\le \min\{d-i,i\}}\left( \sqrt{f_{i+j}f_{i-j}}\right)^{\beta_i} \\
& \leq  2(1 - {a}) + {a} f + \sum_ {(i,j)\in \Psi}\omega _ {i} ^ {\beta _ i}  \omega^{2\beta_i} \left( \sqrt{f_{i+j}f_{i-j}}\right)^{\beta_i}.
\end{align*}

%the proposition, 
Now,  for $(i, j) \in \Psi \smallsetminus \Phi$ it follows that
\begin{equation*}
\big(f _ {i-j} (h) f _ {i + j} (h)\big) ^ {1/ 2 \beta (i)} \leq C_{i,j,\varepsilon} f (h) ^{\frac {\beta(i-j) \beta (i+j)}{2 \beta (i)}} \leq C_{i,j,\varepsilon} \left(1+ f (h) ^ {{{b}}}\right)
\end{equation*}
for some constants $C_{i,j,\varepsilon}$  and all $h \in H$.
%For $(i, j) \in \Phi$ we 
On the other hand, $(i,j) \in \Phi$ implies
$$
d_-(i-j)= d _ - (i) - j = d _ - (i + j) - 2 j
$$
and
$$
d_+(i-j)= d _ + (i) + j = d _ + (i + j) + 2 j.
$$
Hence $\sqrt {\omega_{i+j} \omega_{i-j}} = \omega _ i \varepsilon ^ {j ^2}$
and thus
\begin{align*}
\omega_i^{\beta_i}\big(f _ {i-j} (h) f _ {i + j} (h)\big) ^ {\beta_i/ 2} &\leq 
\left(\frac{\omega_{i}}{\sqrt{\omega_{i-j} \omega _ {i + j}}} \right) ^ {\beta_i} 
 \left (f(h)^{\beta(i-j)}f(h)^{\beta(i+j)}\right)^{\beta_i/2}  \\
& = \varepsilon ^ {\beta_i j^2}f(h)
.\end{align*}

We now estimate separately this sum for $(i,j) \in \Psi \smallsetminus \Phi$ and for $(i, j) \in \Phi$:
%\begin{align*}
$$\sum_ {(i,j) \in \Psi \smallsetminus \Phi} \omega^{2\beta_i} \left( \sqrt{f_{i+j}f_{i-j}}\right)^{\beta_i} \leq C_{\varepsilon, \omega} (1+f^{{{b}}}) \,,
$$
for some constant $C_{\varepsilon, \omega}\,$, and
$$
\sum_ {(i,j) \in \Phi} \omega^{2\beta_i} \left( \sqrt{f_{i+j}f_{i-j}}\right)^{\beta_i} \leq \left (\sum_ {(i,j) \in \Psi} \omega^{2\beta_i} \varepsilon ^ {\beta _ i j ^2} \right)f\,.
$$
%\end{align*}
Recall that we are given an arbitrary ${a}' > {a}$. If we take $\varepsilon$ small enough so that
\begin{equation*}
\sum_ {(i,j) \in \Psi} \omega^{2\beta_i} \varepsilon ^ {\beta _ i j ^2} < \frac {{a}'-{a}}2\,,
\end{equation*}
then, using the fact that $\omega_i^{\beta_i} \le 1$, we can conclude that
\begin{equation*}
\begin{aligned}
Af &\leq 2(1-{a}) + {a}f + C_{\varepsilon, \omega} (1+f^{{{b}}}) +\frac {{a}'-{a}}2 f \\ &< 2 + {a}'f - \frac {{a}'-{a}}2 f +C_{\varepsilon, \omega} (1+f^{{{b}}}) \leq {a}'f+C'_{\varepsilon, \omega}
\end{aligned}
\end{equation*}
for an appropriate constant $C ' _ {\varepsilon, \omega}$.
\end{proof}

\section{Averages of \texorpdfstring{${\tilde \alpha}$}{alpha} under \texorpdfstring{$g_{Nt}$}{the flow} and coverings by small balls}

For the next two statements we will take  $c$  as in Corollary~\ref{cor:main}, and 
for a given $t \ge 1$ fix ${\tilde \alpha}$ as in \eqref{eqn:alphatilde} and $\Mtilde$ as in Corollary~\ref{cor:main}. Also, for any $x \in X$,  $M>0$ and $N \in \N$ let us define
$$Z_x(M,N,t):=\big\{u\in B_1^U:  {\tilde \alpha}(g_{t\ell} u x)  \, > \, M \ \  \,\forall \ell \in \{1,\dots,N\}\big \}.$$
Observe that the function $\tilde \alpha$ is proper due to Mahler's Compactness Criterion, hence for any $M > 0$ the set
$$
X_{\le M} := \{x\in X : \tilde\alpha(x) \le M\}
$$
is compact. Note also that clearly $Z_x(M,N,t) = Z_x(X_{\le M},N,t,1)$, where the latter is defined as in  \equ{zx}.

\begin{prop}\label{prop:iterates} 
%Let $c$ be as in Corollary~\ref{cor:main}, and 
%for a given $t \ge 1$ let ${\tilde \alpha}$ be as in \eqref{eqn:alphatilde} and $\Mtilde$ as in Corollary~\ref{cor:main}. For any $x \in X$,  $M>0$ and $N \in \N$ we define
%$$Z_x(M,N,t):=\big\{u_s\in B_1^U:  {\tilde \alpha}(g_{t\ell} u_s x)  \, > \, M \ \  \,\forall \ell \in \{1,\dots,N\}\big \}.$$
% Then  t
There exists $M_0>\Mtilde$ such that for any $x \notin X_{\le T}$, %with ${\tilde \alpha}(x)  \, > \,  \Mtilde$, 
any $ N \in \N$ and any $M>M_0$ we have
$$\int_{Z_x(M,N-1,t)} {\tilde \alpha}(g_{Nt} u_s x) ds\ll c^N t^{N} e^{-mntN}{\tilde \alpha}(x)\,.$$
%for any $N \in \N$, and some constant $c>0$ independent of $N$, $t $ and $x$.
\end{prop}

\begin{proof}
For any $\sigma>0$ let $\rho_{\sigma^2}$ denote the Gaussian probability measure on $\R^{mn}$ where each component is i.i.d.\ with mean 0 and variance $\sigma^2$. In particular, $\rho=\rho_1$.  We will use the following fact: for any continuous function $f$ on $U$ and $\varepsilon>0$
\begin{equation}\label{eqn:conv}
\int_U \int_U f(\varepsilon x+y) \,d\rho_1(x)\,d\rho_1(y) = \int_U f(z) \,d\rho_{1+\varepsilon^2}(z).
\end{equation}

Let $t \ge 2c$ be given and let ${\tilde \alpha}$, $\Mtilde$ be as 
%in Corollary~\ref{cor:main}
above. 
%As before we let $X_{\ge M}$ be the cusp neighborhood $\{x \in X: {\tilde \alpha}(x) \ge M\}.$
Corollary~\ref{cor:main} gives 
\begin{equation}\label{eqn:51}
\int_U {\tilde \alpha}( g_t u_s x) \,d\rho_1(s)\le c t e^{-mnt}{\tilde \alpha}(x) \text{ for any }x \in X
%_{\ge \Mtilde}\,.
\text{ with } \tilde\alpha(x) > \Mtilde.
\end{equation}

It follows from the definition of $\tilde \alpha$ that there is a constant $C _ {\tilde \alpha}$, dependent only on $m,n$, such that 
\eq{Calpha}{C _ {\tilde \alpha}^{-1} \leq \frac{\tilde \alpha(ux)}{\tilde \alpha(x)} \leq C _ {\tilde \alpha} \qquad\text{for any~$u \in B_2^U$ and $x\in X$}.}
Pick $M \geq C _ {\tilde \alpha} \Mtilde$.

Let $N \in \N$ be given. Using \eqref{eqn:51} repeatedly we get
\begin{equation}\label{eqn:intZ}
\int \dots \int_{Z} {\tilde \alpha}(g_t u_{s_N}\dots g_t u_{s_1} x  ) \,d\rho_1(s_N)\dots \,d\rho_1(s_1) \le c^N t^N e^{-mntN} {\tilde \alpha}(x),
\end{equation}
where 
$$Z=\{(s_1,\dots,s_N) \in U^N: \tilde\alpha(g_t u_k \dots g_t u_1x ) > \Mtilde
%\in X_{\ge \Mtilde} 
\ \  \forall \,k=1, \dots,N-1\}.$$
Write $g_t u_{s_N}\dots g_t u_{s_1} =g_{Nt} u_{\phi(s_1,\dots,s_N)} $ where
$$\phi(s_1,\dots,s_N):=\sum_{k=1}^N e^{-(k-1)(m+n)t} s_k.$$
Thus, using \eqref{eqn:conv} we see that \eqref{eqn:intZ} takes the form
\begin{equation}\label{eqn:sigma2}
\int_{U} 1_{\phi(Z)}(s) {\tilde \alpha}(g_{Nt} u_{s}x  ) \,d\rho_{\sigma^2}(s)\le c^N t^N e^{-mntN} {\tilde \alpha}(x),
\end{equation}
where $\sigma^2=\sum_{k=1}^N e^{-2(k-1)(m+n)t}$. Although $\sigma^2$ depends on $N$, it is in $[1,2]$ whenever $t \ge 1$. This implies that $ds$ is absolutely continuous with respect to  $d\rho_{\sigma^2}$ on $B_1^U$,  with a uniform bound on the Radon-Nikodym derivative. Hence, \eqref{eqn:sigma2} gives
\eq{overtheball}{\int_{B_1^U} 1_{\phi(Z)}(s){\tilde \alpha}(g_{Nt} u_{s}x  ) ds\ll c^N t^N e^{-mntN} {\tilde \alpha}(x).}

We claim that $\phi(s_1,\dots,s_N)\in Z_x(M,N-1,t)$ implies that  $(s_1,\dots,s_N)\in Z$.
% where $s$ is as above. 
%Then, 
Assuming the claim, we get that
$$\int_{Z_x(M,N-1,t)}{\tilde \alpha}(g_{Nt} u_{s}x  ) ds\ll c^N t^N e^{-mntN} {\tilde \alpha}(x).$$
So it suffices to show the claim. Suppose that $$s\,=\phi(s_1,\dots,s_N)=\sum_{k=1}^N e^{-(k-1)(m+n)t} s_k\in Z_x(M,N-1,t)\,.$$ This means that for any $\ell=1,\dots,N-1$ 
%we have  $g_{\ell t}u_s x \in X_{\ge M}$. Therefore for any $\ell=1,\dots,N-1$ 
we have
\begin{equation}\label{eqn:y'y''}
\tilde\alpha(g_{\ell t}u_sx) =\tilde\alpha\big((g_{\ell t}u_{ s'}g_{-\ell t}) g_{\ell t} u_{s''}\big) 
> M,
\end{equation}
where $s'=\sum_{k=\ell +1}^N  e^{-(k-1)(m+n)t} s_k$ and $s''=\sum_{k=1}^\ell  e^{-(k-1)(m+n)t} s_k.$ Clearly
$g_{\ell t}u_{ s'}g_{-\ell t} \in B_2^U$; therefore 
%so that 
\eqref{eqn:y'y''} together with our choice of $M$ give
$$\tilde\alpha(g_{\ell t} u_{s''} x) =\tilde\alpha( g_t u_{\ell} \dots g_t u_{1} x) 
> \Mtilde.$$
Thus  $(s_1,\dots,s_N)\in Z$, which finishes the proof of the claim.
\end{proof}

\begin{cor}\label{cor:coverN}
For all sufficiently large $M>0$, any $N \in \N$ and $x \in X$, the set $Z_x(M,N,t)$ can be 
covered with $\frac{\tilde \alpha(x)}{M} C_1^N t^{N} e^{(m+n-1)mntN}$ balls in $U$ of radius 
$e^{-(m+n)tN}$  for some $ C_1>0$ independent of $t$, $N$ and $x$.
\end{cor}

\begin{proof}
We partition $B _ 1 ^ U$ into $p \leq C e ^{(m+n-1)mntN}$ disjoint subsets $D _ 1$, \dots, $D _ p$, each containing a ball of radius $r/10$ and contained in a ball of radius $r$ for $r=e ^ {- (m + n) t N}$. This can be done e.g.\ by choosing a maximal $r/2$-separated subset of $B _ 1 ^ U$, say $\{u_1,\dots,  u_p\}$, and letting
\[
D_i := \left(B^U_r u_i \smallsetminus \Big(\bigcup _ {j = 1} ^ {i-1} D_j \cup \bigcup_ {j=i+1}^p B^U_{r/2}u_i\Big)\right)\cap B^U _1.
\]
Note that for $u _ i$ near the boundary of $B _ 1 ^ U$ the set $D _ i$ may fail to contain the ball $B _ {r/2} ^ U (x _ i)$, but $B _ {r/2} ^ U u _ i \cap B _ 1 ^ U \subset D _ i$ will certainly contain some ball of radius $r / 10$.

Let $\nu$ denote the Lebesgue measure on $U$, normalized so that $\nu (B _ 1 ^ U) = 1$. By Proposition~\ref{prop:iterates} we know that for $M$ sufficiently large (depending on $t$)
\begin{equation*}
\int_ {Z _ x (C_{\tilde \alpha}^{-1} M, N -1, t)} {\tilde \alpha}(g_{Nt} u_s x) ds\ll c^N t^{N} e^{-mntN}{\tilde \alpha}(x)
\end{equation*}
with $C _ {\tilde \alpha}$ as in the proof of that proposition, hence
\begin{equation*}
\nu \big( Z _ x (C_{\tilde \alpha}^{-1} M, N , t)\big) \ll \frac {\tilde \alpha(x)}{M} {c^N}t^{N} e^{-mntN}
.\end{equation*}
Since $\nu (D _ i) \gg  e^{-mn(m+n)tN}$ for all $i$, it follows that the number $p_1$ of the $D_i$'s that are  contained in $Z _ x (C_{\tilde \alpha}^{-1} M, N , t)$ satisfies $$p_1 \leq \frac {\tilde \alpha(x)}{M} {C_1^N}t^{N} e^{mn(m+n-1)tN}$$ (for an appropriate constant $C _ 1$ independent of $t,N,x$). Reordering the $D_i$'s if necessary we can assume that these are exactly $D_1$, \dots, $D_{p _1}$.

Take now $i> p _ 1$. Then $D_i$ contains at least one element $u$ outside the set $Z _ x (C_{\tilde \alpha} ^{-1}M, N , t)$, therefore
for some $1 \leq \ell \leq N$ it holds that
\(
\tilde \alpha (g _ {t \ell} u x) \leq M / C _ {\tilde \alpha}
\).
But then, as $D_i \subset B _ r ^Uu_i$,
\begin{equation*}
g _ {t \ell} D_i \subset g_ {t \ell} B _ {2 r} ^ U u = B _ {e^{(m + n) t \ell} 2r} ^ U g _ {t \ell} u.
\end{equation*}
Hence, since 
$e^{(m + n) t \ell} 2r \leq 2$, by definition of $C _ {\tilde \alpha}$ one has
\(
g _ {t \ell} D_i x \subset X _{\leq M,}
\)
so $D_i$ is disjoint from $Z _x (M,N,t)$. Thus
\begin{equation*}
Z _x (M, N, t) \subset \bigcup_ {i = 1} ^ {p _ 1} D_i,
\end{equation*}
and the proposition follows.
\end{proof}

\begin{proof}[Proof of Theorem~\ref{thm:main}]
Let $t > t_0$ be given, with $t_0$ a large real number, to be determined later, depending only on $m,n$. Let ${\tilde \alpha}$ and $\Mtilde$ be as in Corollary~\ref{cor:main}. 
We will find large enough $M > 0$ such that the compact set $Q=X_{ \le M}$ satisfies the conclusion of the theorem. 

For a given $N \in \N$ and $x \in X$ we consider a subset $J_x$ of $ \{1,\dots,N\}$ given by
$$J_x:=\big\{\ell \in  \{1,\dots,N\} : g_{\ell t} x 
\notin Q\big\}.$$
Then one can write the set $Z_x(Q,N,t,\delta)$ as
$$Z_x(Q,N,t,\delta)=\{u \in B_{1}^U : |J_{u x }| \ge \delta N \}.$$

For any  subset $J$ of $ \{1,\dots,N\}$ we set 
$$Z(J):=\{u \in B_{1}^U : J_{ u x}=J\}.$$
We note that $Z_x(Q,N,t,\delta) =\bigcup_J Z(J)$ where the union runs over all subsets $J$ 
of $ \{1,\dots,N\}$ with cardinality  at least $ \delta N$. Clearly, the number of such subsets 
of $ \{1,\dots,N\}$ is at most $2^N \le t^N.$ Thus, it suffices to show that for a given subset 
$J \subset  \{1,\dots,N\}$, the set $Z(J)$ can be covered with $C(x)t^{2N}e^{mnt[(m+n)N-|J|]}$ balls of radius $e^{-(m+n)tN},$ for
\eq{cx2}{
C(x)=\max\{\alpha_i^{\beta_i}(x):i=1,2,\dots,m+n-1\},
}
cf.~\equ{cx}.

Let $J$ be as above. We decompose $J$ into ordered subintervals $J_1, \dots,J_p$ of maximal possible sizes such that $J=\bigsqcup_{i=1}^p J_i.$ Let $I_1, \dots,I_{p'}$ be the ordered maximal subintervals of $ \{1,\dots,N\} \smallsetminus J$ such that 
$$ \{1,\dots,N\}=\bigsqcup_{i=1}^p J_i \sqcup \bigsqcup_{j=1}^{p'} I_j.$$
We now inductively prove the following claim: for any integer $L \le N$, if 
 \begin{equation}\label{eqn:indK}
 \{1,\dots,L\}=\bigsqcup_{i=1}^\ell J_i \sqcup \bigsqcup_{j=1}^{\ell'} I_j,
 \end{equation} then the set $Z(J)$ can be covered with 
 \begin{equation}\label{eqn:S}
 S \le\max\left(1, \frac{\tilde\alpha(x)}{M}\right)t^{2L}e^{mnt[(m+n)L-|J_1|-\cdots - |J_{\ell}|]}
 \end{equation}
  sets of the form $D u_1 ,  \dots, D u_S $ where $D=g_t^{-L} B_\eta^U g_t^{L},$ i.e.\ a ball of radius $e^{-(m+n)tL}$. Comparing (\ref{eqn:alphatilde}) and \equ{cx2}, we note that for sufficiently large $M$ we have 
  \eq{cxbound}{\max\left(1, \frac{\tilde\alpha(x)}{M}\right) \le C(x).}
  Thus,
   by letting $L=N$ we establish the 
   claim for $|J|\ge \delta N$. If in the first step we have $\{1,\dots,L\}=J_1$  then \eqref{eqn:S} follows from Corollary~\ref{cor:coverN} once $t \ge C_1$ and $M$ is large enough to satisfy the conclusion of the corollary. If $\{1,\dots,L\}=I_1$ then as $Z(J) \subset B_1^U$ it is obvious that the set $Z(J)$ can be covered with 
   $\leq C_2 e^{mnt(m+n)L}$ balls of radius $e^{-(m+n)L}$, with $C_2$ depending only on $nm$. 
   
   Assume now the set $Z(J)$ can be covered with $S$ balls of radius $ e^{-(m+n)tL}$ for some $L$ satisfying \eqref{eqn:indK}. In the inductive step, for the next $L' >L$ satisfying an equation similar to \eqref{eqn:indK}, we have two cases:  either \eq{case1}{ \{1,\dots,L'\}= \{1,\dots,L\} \sqcup I_{\ell'+1}} or \eq{case2}{ \{1,\dots,L'\}= \{1,\dots,L\} \sqcup J_{\ell+1}.}
 
Consider first the case
 \equ{case1}. Obviously, each box $ Du_i$ in $U$ of size $\eta e^{-(m+n)tL}$ can be covered by $C_2e^{(m+n)t|I_{\ell'+1}|mn}$ balls of radius 
 $e^{-(m+n)t(L+|I_{\ell'+1}|)}$.
 Thus, noting that $L+|I_{\ell'+1}|=L'$, from \eqref{eqn:S} and  \equ{cxbound} it follows that if we assume, as we may, that $t>C_2$, the set $Z(J)$ can be covered by 
 $$C(x) t^{2L'}e^{mnt[(m+n)L'-|J_1|-\cdots - |J_{\ell}|]}$$
  balls of radius $\eta e^{-(m+n)tL'}$ as claimed.

 Now assume \equ{case2} and 
 consider one of the balls $ D u_i$ of radius $e^{-(m+n)tL}.$ We are interested in bounding the number of balls of radius $e^{-(m+n)t(L+|J_{\ell+1}|)}$ needed to cover $Z(J) \cap D u_i$. If 
 $Z(J) \cap D u_i = \varnothing$ there is nothing to cover. 
 
 So let $u \in Z(J) \cap Du_i$. By definition of $Z(J)$ this implies
$\tilde\alpha(g_t^L ux) \leq M$; on the other hand
$\tilde\alpha(g_t^j ux)>M$ for all $j \in J _{\ell+1}$. Since $g _ t$ expands every vector in $\bigwedge  \R ^ {m+n}$ by at most $e^{C_3 t}$, with $C _ 3$ depending only on $m,n$,
it follows that
$$\tilde\alpha(g_t^L u_i x) \geq e^{-C_3 t}M.$$ Hence using \equ{Calpha} one gets
\begin{equation*}
C_{\tilde \alpha} M \geq \tilde \alpha (g _ t ^ L x) \geq C_{\tilde \alpha} ^{-1} e ^ {- C _ 3 t} M
.\end{equation*}
Assuming $M$ is large enough (depending on $t$) so that
Corollary~\ref{cor:coverN} is applicable to  $x'= g_t^{L} u_i x$, we see that $Z _ {x '} (M, |J_{\ell +1}|, t)$ can be covered by 
\[
C_{\tilde\alpha} C_1^{|J_{\ell +1}|} t^{|J_{\ell +1}|} e^{(m+n-1)mnt|J_{\ell +1}|} \leq t^{2|J_{\ell +1}|} e^{(m+n-1)mnt|J_{\ell +1}|}\]
 balls in $U$ of radius $e^{-(m+n)|J_{\ell +1}|}$ (assuming $t_0>C_{\tilde\alpha} C_1 $).

 Note that by definition of $Z(J)$ and $ Z_x({\cdot}, {\cdot}, {\cdot})$ one has
\[
Z (J) \cap D u _ i \subset g_t^{-L} Z_{x'}(M, |J_{\ell +1}|, t) g_t^{L}
.\]
Our bound on the number of $e^{-(m+n)|J_{\ell +1}|}$-balls needed to cover $Z_{x'}(M, |J_{\ell +1}|, t) $ implies that $g_t^{-L} Z_{x'}(M, |J_{\ell +1}|, t) g_t^{L}$ can be covered by at most \[t^{2|J_{\ell +1}|} e^{(m+n-1)mnt|J_{\ell +1}|}\] balls of radius \(e^{-(m+n)(L+|J_{\ell +1}|)}=e^{-(m+n)(L ')},\) hence $Z(J)$ can be covered by \[t^{2L'} e^{mnt[(m+n)L'-|J_1| - \cdots - |J_{\ell+1}|]}\]
balls of radius $e^{-(m+n)(L ')}$, establishing the inductive hypothesis.
\end{proof}


\begin{thebibliography}{EMLV13}

\bibitem[Ath06]{Athreya}
J.\ Athreya.
\ \emph{Quantitative recurrence and large deviations for {T}eichmuller
  geodesic flow},
Geom.\ Dedicata {\bf 119} (2006), 121--140.

\bibitem[BK83]{BK83}   M.~Brin and A.~Katok,
     \emph{On local entropy}, in: Geometric dynamics (Rio de Janeiro, 1981),  Lecture Notes in Math., {\bf 1007}, Springer, Berlin, 1983, pp.\ 30--38.

\bibitem[Che11]{Che11}   Y.~Cheung,
     \emph{Hausdorff dimension of the set of singular pairs}, Ann.\ Math.\ {\bf 173} (2011), 127--167.
 
 \bibitem[CC14]{CC14}   Y.~Cheung and N.~Chevallier,
     \emph{Hausdorff dimension of singular vectors}, preprint,  2014.
      
\bibitem[D85]{Dani}  S.G.\ Dani, \emph{Divergent trajectories of flows on
\hs s and Diophantine approximation}, 
J.\ Reine Angew.\ Math.\ {\bf 359} (1985), 55--89. 

  \bibitem[EK12]{EK12}   M.\ Einsiedler and S.\ Kadyrov,
     \emph{Entropy and escape of mass for $\SL(3, \Z)\backslash \SL(3, \R)$}, Israel J.\ Math.\ {\bf 190} (2012), 253--288.
 
  \bibitem[EKP13]{EKP13}   M.~Einsiedler, S.~Kadyrov and A.D.~Pohl,
     \emph{Escape of mass and entropy for diagonal flows in real rank one situations}, preprint,  \url{http://arxiv.org/abs/1110.0910v2}.
      
 \bibitem[ELMV12]{ELMV13}   M.~Einsiedler, E.~Lindenstrauss, P.~Michel  and A.~Venkatesh,
     \emph{The distribution of periodic torus orbits on homogeneous spaces, II: Duke's theorem for quadratic fields},  Enseign.\ Math.\  {\bf 58} (2012), 249--313.

\bibitem[EMM98]{EMM}    A.\ Eskin,  G.A.\ Margulis  and S.\ Mozes,
     \emph{Upper bounds and asymptotics in a quantitative version of the
              {O}ppenheim conjecture}, Ann.\ Math.\ \textbf{147} (1998), no.\ 2, 93--141.   


\bibitem[EM11]{Eskin-Mirzakhani}
A.~Eskin and M.~Mirzakhani, 
 \emph{Counting closed geodesics in moduli space},
J.\ Mod.\ Dyn.\ {\bf 5} (2011), no.\ 1, 71--105.			  


 \bibitem[EW11]{EW11}
M.~Einsiedler and T.~Ward, \emph{Ergodic Theory with a view towards Number Theory}, Graduate Texts in Mathematics, {\bf 259},  Springer-Verlag, London, 2011.
 
\bibitem[Ham11]{Hamendstadt}
U.~Hamenst\"adt, \emph{Symbolic dynamics for the Teichmueller flow}, preprint, \url{http://arxiv.org/abs/1112.6107}.

\bibitem[Kad12a]{Kad12(a)}  S.~Kadyrov,
     \emph{Positive Entropy Invariant Measures on the Space of Lattices with Escape of Mass}, Ergodic Theory Dynam.\  Systems {\bf 32} (2012), no.\ 1, 141--157.


\bibitem[Kad12b]{Kad12(b)}  S.~Kadyrov,
     \emph{Entropy and Escape of Mass for Hilbert Modular Spaces}, Journal of Lie Theory {\bf 22} (2012), no.\ 3, 701--722.
     
   \bibitem[Kh26]{Kh} A.\ Khintchine, \emph{\"Uber eine klasse linearer diophantische approximationen}, Rendiconti Circ.\ Mat.\ Soc.\
Palermo {\bf 50} (1926), 170--195.
     
     
  \bibitem[KP12]{KP12}   S.~Kadyrov and A.D.~Pohl,
     \emph{Amount of failure of upper-semicontinuity of entropy in noncompact rank one situations, and Hausdorff dimension}, preprint,  \url{http://arxiv.org/abs/1211.3019}.

 \bibitem[Wal65]{Wal65}
P.\ Walters,
\emph{An introduction to ergodic theory}, 
Springer-Verlag, 1965.  
\end{thebibliography}
\end{document}